\documentclass{amsart}

\usepackage{stackrel} 
\usepackage{a4}
\usepackage{amssymb,amsmath}	
\usepackage[all]{xy}		
\usepackage{stmaryrd}		
\usepackage{hyperref}		

\numberwithin{equation}{section}	
\newtheorem{theorem}{Theorem}[section]
\newtheorem{lemma}[theorem]{Lemma}
\newtheorem{proposition}[theorem]{Proposition}

\theoremstyle{definition}
\newtheorem{construction}[theorem]{Construction}

\newtheorem{remark}[theorem]{Remark}

\theoremstyle{remark}
\newtheorem*{acknowledgement}{Acknowledgement}

\newcommand{\KK}{\mathbb K}
\newcommand{\PP}{\mathbb P}
\newcommand{\ZZ}{\mathbb Z}
\newcommand{\gq}{/\!\! /}

\title{On the Cox ring of blowing up the diagonal}
\author{Hendrik B\"aker}
\address{Mathematisches Institut,
Universit\"at T\"ubingen,
Auf der Morgenstelle 12,
72076 T\"ubingen,
Germany}
\email{hendrik.baeker@mathematik.uni-tuebingen.de}

\begin{document}

\begin{abstract}
We compute the Cox rings of the blow-ups $\mathrm{Bl}_\Delta(X'\times X')$ and $\mathrm{Bl}_\Delta(\PP_1^n)$ where $X'$ is a product of projective spaces and $\Delta$ is the (generalised) diagonal.
\end{abstract}

\maketitle

\section{Introduction}
To a complete, normal variety with finitely generated divisor class group one can 
associate its Cox ring. In recent literature it has been discussed how the Cox ring behaves under blow-ups. In particular, it is of interest whether finite generation is preserved in this process, and, if so what a presentation in terms of generators and relations looks like, see for example
\cite{cox_ring_m0134, blowing_up_pn, blowing_up_weighted_pn, computing_cox_rings, collineations, points_on_line}.

In the present note we employ the techniques developed in \cite{on_chow_quotients, computing_cox_rings} to compute the Cox rings of the following blow-ups. Let $X':=\PP_{n_1}\times\ldots\times\PP_{n_{\mathbf r}}$ be a product of projective spaces and denote by ${\Delta_X}\subseteq X:= X'\times X'$ the diagonal. The variety $X$ is spherical and $\mathrm{Bl}_{\Delta_X}(X)$ inherits this property. Hence, it is known that the Cox ring $\mathcal R(\mathrm{Bl}_{\Delta_X}(X))$ is finitely generated, see \cite{cox_rings,spherical_varieties}. Our first result is an explicit presentation.

\begin{theorem}
\label{thm:main1}
The Cox ring $\mathcal R(\mathrm{Bl}_{\Delta_X}(X))$ of the blow-up $\mathrm{Bl}_{\Delta_X}(X)$ is isomorphic to the $\ZZ^{\mathbf r}\times\ZZ^{\mathbf r}\times\ZZ$-graded factor algebra $R_X/I_X$ where
\begin{align*}
 R_X\ :&=\ \KK[T_\infty,\ _rT_{ij};\quad r=1,\ldots, \mathbf r,\quad 0\le i<j\le n_r+2,\quad i\le n_r],\\
 I_X\ :&=\  I(1)+\ldots+I(\mathbf r),
\end{align*}
for every $r=1,\ldots,\mathbf r$ the ideal $I(r)$ is generated by the twisted Pl\"ucker relations
\begin{align*}
  {_rT_{ij}}\ {T_\infty}\ -\ {_rT_{ik}}\ {_rT_{jk}}\ +\  {_rT_{il}}\ {_rT_{jk}};
  &\quad 
  0\le i<j\le n_r,\quad k=n_r+1,\ l= n_r+2,\\
   {_rT_{ij}}\ {_rT_{kl}}\ -\ {_rT_{ik}}\ {_rT_{jk}}\ +\ {_rT_{il}}\ {_rT_{jk}};
  &\quad
  0\le i<j< k<l\le n+2,\quad k\le n_r,
 \end{align*}
and the grading of $R_X/I_X$ is given by
\[
 \deg\, (T_\infty) \ =\ (0,0,1),\qquad 
 \deg\, (_rT_{ij}) \ =\ \begin{cases}
                            (e_r,0,0) &\text{if }j=n_r+1,\\
                            (0,e_r,0)&\text{if }j=n_r+2,\\
                            (e_r,e_r,-1)&\text{else}.
                           \end{cases}
\]

\end{theorem}

In particular, the spectrum of the Cox ring $\mathcal R(\mathrm{Bl}_{\Delta_X}(X))$ is the intersection of a product of affine Grassmannian varieties (w.r.t. the Pl\"ucker embedding) with a linear subspace.

As a second class of examples we treat the (non-spherical) blow-up of the variety $Y:=\PP_1^n$ in the generalised diagonal ${\Delta_Y}:=\{(x,\ldots,x);~x\in\PP_1\}\subseteq Y$. Again we prove that the Cox ring of $\mathrm{Bl}_{\Delta_Y}(Y)$ is finitely generated and we give an explicit presentation. 

\begin{theorem}
\label{thm:main2}
 The Cox ring $\mathcal R(\mathrm{Bl}_{\Delta_Y}(Y))$ of the blow-up $\mathrm{Bl}_{\Delta_Y}(Y)$ is isomorphic to the $\ZZ^{n+1}$-graded factor algebra $R_Y/I_Y$ where
 \begin{align*}
  R_Y\quad &:=\quad
  \KK[S_{ij};~1\le i<j\le n+2]\\
  I_Y\quad &:=\quad
  \langle\, S_{ij}S_{kl}\ -\ S_{ik}S_{jl}\ +\ S_{il}S_{jk};\quad 1\le i<j<k<l\le n+2\,\rangle,
 \end{align*}
 and the grading of $R_Y/I_Y$ is given by
\[
 \deg\, (S_{ij}) \ =\ \begin{cases}
			    e_{i} &\text{if }i\le n,\,j=n+1,\, n+2,\\
			    e_{n+1}&\text{if }i=n+1,\, j=n+2,\\
                            e_{i}+e_{j}-e_{n+1}&\text{else}.
                           \end{cases}
\]
\end{theorem}

\begin{acknowledgement}
I would like to thank J\"urgen Hausen for helpful discussions and suggestions.
\end{acknowledgement}

\section{Proofs of Theorems~\ref{thm:main1} and \ref{thm:main2}}

Let us recall some definitions, for details see \cite{cox_rings}. Let $Z$ be a normal variety with free and finitely generated divisor class group $K:=\mathrm{Cl}(Z)$ and only constant invertible regular functions. Then we define its {\it Cox ring} as the $K$-graded $\KK$-algebra
\[
 \mathcal{R}(Z) 
\quad :=\quad
\bigoplus_{K} \Gamma(Z,\mathcal{O}(D)).
\]
If the Cox ring $\mathcal R(Z)$ is finitely generated, we call its spectrum $\overline Z:=\mathrm{Spec}(\mathcal R(Z))$ the {\it total coordinate space} of $Z$. The $K$-grading of $\mathcal R(Z)$ gives rise to an action of the quasitorus $H_Z:=\mathrm{Spec}(\KK[K])$ on $\overline Z$. Moreover, there exists an open invariant subset, the {\it characteristic space}, $\hat Z\subseteq\overline Z$ admitting a good quotient $p_Z\colon \hat Z\to \hat Z\gq H_Z\cong Z$ for this action.

Before we enter the proofs we will sketch the methods developed in \cite{on_chow_quotients, computing_cox_rings}. Let $Z$ be a toric 
variety with Cox construction $p_{Z}\colon\hat Z\to Z$, total coordinate space $\overline Z=\KK^{r}$ and an ample class $w\in K$ in the divisor class group. 

Now let $\mathfrak A$ be a subscheme of $Z$; we ask for the Cox ring of the blow-up $\mathrm{Bl}_{\mathfrak A}(Z)$ of $Z$ in $\mathfrak A$. By Cox' construction \cite{homogeneous_coordinate_ring} the subscheme $\mathfrak A$ arises from a homogeneous ideal $\mathfrak a=\langle f_1,\ldots,f_l\rangle$ in the $K$-graded Cox ring $\mathcal R(Z)$. For this consider the associated $K$-graded sheaf $\tilde{\mathfrak a}$ on $\overline Z$; then $\mathfrak A$ is given by $({p_Z}_*\tilde{\mathfrak a})_0$. As a first step we embedd $Z$ into a larger toric variety $Z_1$ such that the blow-up can be dealt with using methods from toric geometry. For this we consider the closed embedding
\[
 \overline \pi\colon\KK^{r}\quad \to\quad \KK^{r_1};
 \qquad 
 z\quad\mapsto\quad (z,f_1(z),\ldots,f_l(z)),
\]
where $r_1:=r+l$. We endow $\KK[T_1,\ldots,T_{r_1}]$ with a grading of $K_1:=K$ by assigning to $T_1,\ldots,T_r$ the original $K$-degrees and setting $\deg(T_{r+i}):=\deg f_i$ for the remaining variables. Then the quasitorus $H_{Z_1}:=\mathrm{Spec}(\KK[K_1])$ acts on the affine space $\overline Z_1:=\KK^{r_1}$ and this makes $\overline \pi$ equivariant. The class $w\in K_1$ gives rise to an open subset $\hat Z_1\subseteq\overline Z_1$ and a toric variety $Z_1:=\hat Z_1\gq H_{Z_1}$.

The closed embedding $\overline \pi$ restricts to a closed embedding $\hat\pi\colon\hat Z\to\hat Z_1$ of the corresponding characteristic spaces and then descends to a closed embedding $\pi\colon Z\to Z_1$ of the respective quotients. The setting fits into the following commutative diagram.

\[
 \xymatrix{
 {\overline Z}
 \ar[rr]^{\overline\pi}
 &&
 {\overline Z_1}
 \\
 {\hat Z}
 \ar[rr]^{\hat\pi}
 \ar[d]^{\gq H_Z}
 \ar[u]
 &&
 {\hat Z_1}
 \ar[d]^{\gq H_{Z_1}}
 \ar[u]
 \\
 Z
 \ar[rr]^\pi
 &&
 Z_1
 }
\]

The idea is to compute the Cox ring of the proper transform $Z'$ of $Z\subseteq Z_1$ with respect to a toric blow-up of $Z_1$. The following lemma relates $Z'$ to the blow-up $\mathrm{Bl}_\mathfrak A(Z)$. Although the result was to be expected, we do not know of a reference and provide a proof.

\begin{lemma}
\label{lem:blow_vs_cox}
Let $\mathfrak b\subseteq\mathcal O(\overline Z_1)$ be a $K_1$-homogeneous ideal and let $\mathfrak B$ be the corresponding subscheme of $Z_1$. Then the proper transform of $Z\subseteq Z_1$ under the blow-up $\mathrm{Bl}_{\mathfrak B}(Z_1)\to Z_1$ is isomorphic to the blow-up of $Z$ in the subscheme of $Z$ associated to the $K$-homogeneous ideal $\overline \pi^*\mathfrak b\subseteq\mathcal O(\overline Z)$. 
\end{lemma}

\begin{remark}
 If we apply Lemma~\ref{lem:blow_vs_cox} in the case $\mathfrak b:=\langle\, T_{r+1},\ldots,T_{r_1}\,\rangle$, then we obtain $\mathfrak a=\overline\pi^*\mathfrak b$ and the proper transform $Z'$ of $Z\subseteq Z_1$ is the blow-up of $Z$ in $\mathfrak A$. Moreover, if $\mathfrak a$ is prime, then the associated subscheme $\mathfrak A$ is the subvariety $p_Z(V(\mathfrak a))$ and $Z'$ is the ordinary blow-up of $Z$ in $p_Z(V(\mathfrak a))$.
\end{remark}

\begin{proof}[Proof of Lemma~\ref{lem:blow_vs_cox}]
First blow-ups are determined locally. We consider a suitable partial open cover of $\hat Z_1$ and of the characteristic space $\hat Z$. Let $w\in K=K_1$ be an ample class of $Z$ as above. We set
\[
 \Gamma:=\big\{
 \gamma\in \{0,1\}^{r_1};\quad
 w\ \in\ \mathrm{relint}\,(\, 
 \mathrm{cone}\left(\deg T_i;~\text{where }\gamma_i=1\right)
 \,)
 \big\}.
\]
Then $\hat Z_1$ is covered by the $H_{Z_1}$-invariant sets ${\overline Z_1}_\gamma\,:=\,{\overline Z_1}\setminus V(T^\gamma)$ where $\gamma\in\Gamma$. We now determine a partial cover which already contains $\overline \pi(\hat Z)$. For this we consider the subset 
\[
 \Gamma'\ :=\ \Gamma\cap(\{0,1\}^r\times\{0\}^l)\ \subseteq\ \Gamma.
\]
Then the corresponding open subvarieties cover the image of $\overline \pi$. More precisely, if we set $\overline Z_\gamma\ :=\ \overline \pi^{-1}({\overline Z_1}_\gamma)$, then we have 
\[
 \hat Z=\bigcup_{\gamma \in\Gamma'}\overline Z_\gamma
 \quad\text{and hence}\quad
 \overline \pi(\hat Z)\subseteq\bigcup_{\gamma\in\Gamma'}{\overline Z_1}_\gamma.
\]
Moreover, we denote by $Z_\gamma\ :=\ \overline Z_\gamma\gq H_Z$ and ${Z_1}_\gamma\ :=\ {\overline Z_1}_\gamma\gq H_{Z_1}$ the respective quotient spaces and fix some $\gamma\in\Gamma'$. If we set $I_1\subseteq \KK[T_1,\ldots,T_{r_1}]$ as the ideal generated by all the $T_{r+i}-f_i$, then the image of $\overline \pi$ is given by $V(I_1)$. The morphism $\overline \pi$ factors into an isomorphism $\overline\pi'\colon \overline Z\to V(I_1)$ and a closed embedding $\iota\colon V(I_1)\to\overline Z_1$.

On the algebraic side we set $A:=\mathcal O(\overline Z)$ and $B:=\mathcal O(\overline Z_1)$. We write $B_\gamma,\,(B/I)_\gamma$ and $A_\gamma$ for the localised algebras and $B_{(0)},\, (B/I)_{(0)}$ and $A_{(0)}$ for their respective homogeneous components of degree zero. Then the situation fits into the following commutative diagrams.
\[
\xymatrix{
{\overline Z}
\ar[r]^(.5){\overline \pi'}
&
{V(I_1)}
\ar[r]^(.5){\overline\iota}
&
{\overline Z_1}
\\
{\overline Z_\gamma}
\ar[r]^(.3){\pi'_\gamma}
\ar[u]
\ar[d]
&
{V(I_1)\cap {\overline Z_1}_\gamma}
\ar[u]
\ar[d]
\ar[r]^(.6){\iota_\gamma}
&
{\overline{Z_1}_\gamma}
\ar[u]
\ar[d]
\\
{Z_\gamma}
\ar[r]^(.5){\pi'}
&
{Z_{\gamma}}
\ar[r]^(.5){\iota}
&
{{Z_1}_\gamma}
}
\qquad\quad
\xymatrix{
  A
  \ar[d]
  &
  {B/I_1}
  \ar@{<->}[l]_(.55){\overline \pi'^*}
  \ar[d]
  &
  B
  \ar@{->>}[l]_(.38){\overline \iota^*}
  \ar[d]
  \\
  {A_\gamma}
  &
  {(B/I_1)_\gamma}
  \ar@{<->}[l]_(.55){{\pi'}_\gamma^*}
  &
  {B_\gamma}
  \ar@{->>}[l]_(.38){\iota_\gamma^*}
  \\
  {A_{(0)}}
  \ar[u]
  &
  {(B/I_1)_{(0)}}
  \ar[u]
  \ar@{<->}[l]_(.55){\pi'^*}
  &
  {B_{(0)}}
  \ar@{->>}[l]_(.38){\iota^*}
  \ar[u]
  }
 \]
The proper transform of $Z_\gamma\subseteq {Z_1}_\gamma$ is the blow-up of $Z_\gamma$ with center given by the affine scheme associated to the ideal $\iota^*\mathfrak b_{(0)}\subseteq (B/I_1)_{(0)}$. Our assertion then follows from the fact that in $A_{(0)}$ the ideals $\pi'^*(\iota^*\mathfrak b_{(0)})$ and $(\overline \pi^*\mathfrak b)_{(0)}$ coincide.
\end{proof}

\begin{construction}
 \label{con:transfer_ideals}
Let $\mathfrak b$ be the ideal $\langle\,T_{r+1},\ldots,T_{r_1}\,\rangle$ and $Z'\to Z$ the proper transform of $Z\subseteq Z_1$ with respect to the toric blow-up $\mathrm{Bl}_{\mathfrak B}(Z_1)\to Z_1$. We turn to the problem of determining the Cox ring $\mathcal R(Z')$.
For this we set $r_2:=r_1+1$ and consider the $r_1\times r_2$-matrix 
\[
A:=[E_{r_1},\mathbf 1_l],
 \qquad\text{where}\quad
 \mathbf 1_l:=(\underbrace{0,\ldots,0}_{r},\underbrace{1\ldots, 1}_{l})^t.
\]
The dual map $A^* \colon \ZZ^{r_1} \to \ZZ^{r_2}$ yields a homomorphism $\alpha^*$ of group algebras and a morphism $\alpha\colon (\KK^*)^{r_2}\to(\KK^*)^{r_1}$. Together with the canonical embeddings $\iota_1^*$ and $\iota_2^*$ we now have transfer the ideal 
\[
 I_1\ :=\ \langle\, T_{r+i}-f_i;~i=1,\ldots,l\,\rangle\quad\subseteq\quad\KK[T_1,\ldots,T_{r_1}]
\]
by taking extensions and contractions via the construction
\[ 
\xymatrix{
I_1
\ar[d]
&
{I_1'\,:=\,\langle\, \iota_1^*I_1\,\rangle}
\ar[d]
&
{I_2'\,:=\,\langle\, \alpha^*I'_1\,\rangle}
\ar[d]
&
{I_2\,:=\,{\iota_2^*}^{-1}I_2'}
\ar[d]
\\
{\KK[T_1, \ldots, T_{r_1}]}
\ar[r]^-{\iota^*_1}
&
{\KK[T_1^{\pm 1}, \ldots, T_{r_1}^{\pm 1}]}
\ar[r]^-{\alpha^*}
&
{\KK[T_1^{\pm 1}, \ldots, T_{r_2}^{\pm 1}]}
&
{\KK[T_1, \ldots, T_{r_2}]}
\ar[l]_-{\iota^*_2}
}
\]
and call the resulting ideal $I_2$. If we endow $\KK[T_1,\ldots,T_{r_2}]$ with the grading of $K_2:=K_1\times \ZZ$ given by 
\[
 \deg (T_i):=\begin{cases}
              (\deg_{K_1}(T_i),0)&\text{for }1\le i\le r,\\
              (\deg_{K_1}(T_i),-1)&\text{for }r+1\le i\le r_1,\\
              (0,1)&\text{for }i=r_2,
             \end{cases}
\]
then $I_2$ is $K_2$-homogeneous and the following Proposition provides us with a criterion to show that $I_2$ defines the desired Cox ring.
\end{construction}

\begin{proposition}[{\cite[Proposition 3.3,~Corollary 3.4]{on_chow_quotients}}]
\label{pro:cox_crit}
If in the $K_2$-graded ring $R_2:= \KK[T_1,\ldots,T_{r_2}]/I_2$ the variable $T_{r_2}$ is prime and does not divide a $T_i$ with 
$1\le i \le r_1$, then $R_2$ is  Cox ring of the proper transform $Z'$.
\end{proposition}

We return to our two cases of $X=X'\times X'$ and $Y=\PP_1^n$. Both of them are toric varieties, their respective Cox rings are polynomial rings and the total coordinate spaces are
\[
  \overline X
 \ =\ 
\bigoplus_{r=1}^\mathbf r \,(\,{\KK^{n_r+1}\oplus \KK^{n_r+1}}\,)
 \qquad\text{and}\qquad
\overline Y
\ =\ 
\underbrace{\KK^{2}\oplus\ldots\oplus\KK^2}_{n}.
\]
On $\overline X$ we will label the coordinates of the $r$-th factor with $_rT_{ij}$ where $i=0,\ldots,n_r$ and $j=n_r+1,n_r+2$. On $\overline Y$ we will use the notation $S_{ij}$ for the coordinates where similarly $i=1,\ldots,n$ and $j=n+1,n+2$.

The first step is to determine generators for the vanishing ideals of the generalised diagonals $\Delta_X$ and $\Delta_Y$ in the respective Cox rings, i.e. the ideals 
\[
 \mathfrak a_X
 \ :=\ 
 I(p_{X}^{-1}(\Delta_{X}))
 \ \subseteq\ 
 \mathcal O(\overline X)
 \qquad\text{and}\qquad
 \mathfrak a_Y
 \ :=\ 
 I(p_{Y}^{-1}(\Delta_{Y}))
 \ \subseteq\ 
 \mathcal O(\overline Y).
\]

\begin{lemma}
\label{lem:vanishing_ideals}
As above let $\mathfrak a_X$ and $\mathfrak a_Y$ be the ideals of the generalised diagonals $\Delta_X$ and $\Delta_Y$ in the respective Cox rings. Both of them are prime and they are generated by the following elements.
\begin{enumerate}
 \item The ideal $\mathfrak a_X$ is generated by the $2\times 2$-minors of the matrices
\[
  \left[
  \begin{array}{cccc}
   _rT_{0,n_r+1}
   &
    _rT_{1,n_r+1}
    &
    \cdots
    &
     _rT_{n_r,n_r+1}
     \\
    _rT_{0,n_r+2}
   &
    _rT_{1,n_r+2}
    &
    \cdots
    &
     _rT_{n_r,n_r+2} 
  \end{array}
  \right],
   \quad
  r=1,\ldots,\mathbf r.
 \]
 \item The ideal $\mathfrak a_Y$ is generated by the $2\times 2$-minors of the matrix
 \[
  \left[
  \begin{array}{cccc}
   S_{1,n+1}
   &
    S_{2,n+1}
    &
    \cdots
    &
     S_{n,n+1}
     \\
    S_{1,n+2}
   &
    S_{2,n+2}
    &
    \cdots
    &
     S_{n,n+2} 
  \end{array}
  \right].
 \]

 \end{enumerate}
\end{lemma}

The idea of the proof is to execute the computations on the respective tori. For future reference let us make the following remark.
\begin{remark}
\label{rem:prime_saturation}
 Let $\iota\colon (\KK^*)^n\to \KK^n$ be the canonical open embedding and $\iota^*$ its comorphism. If $I\subseteq\KK[T_1,\ldots,T_n]$ a prime ideal not containing any of the variables $T_i$, then $(\iota^*)^{-1}\langle \iota^*(I)\rangle =I$ holds.
\end{remark}

\begin{proof}[Proof of Lemma~\ref{lem:vanishing_ideals}]
Let $p_Z\colon \hat Z\to Z$ be the Cox construction of a toric variety $Z$ and $\overline Z$ its total coordinate space. We view the toric morphism $p_Z$ as a mophism $T_{\hat Z}\to T_Z$ of the openly embedded dense tori. Moreover, we denote by $\Delta\subseteq Z$ a subvariety with $\Delta=\overline {\Delta\cap T_Z}$ and write $\iota'\colon \hat Z\to\overline Z$ and $\iota\colon T_{\hat Z}\to \overline Z$ for the canonical open embeddings.
\[
 \xymatrix{
 {T_{\hat Z}}
 \ar@/^0.6cm/[rr]^{\iota}
 \ar[r]
 \ar[d]^{p_Z}
 &
 {\hat Z}
 \ar[d]^{p_Z}
 \ar[r]^{\iota'}
 &
 {\overline Z}
 \\
 T_Z
 \ar[r]
 &
 Z
 &
 \Delta
 \ar@{}[l]|{\supseteq}
 }
\]
Let $\mathfrak d\subseteq\mathcal O(T_Z)$ be the vanishing ideal of $\Delta\cap T_Z$. For the vanishing ideal of $\Delta$ in the Cox ring we obtain 
\[
 I(\iota'(p_Z^{-1}(\Delta)))
 \ =\ 
 I(\overline{\iota(p_Z^{-1}(\Delta\cap T_Z))})
 \ =\ 
 \sqrt{(\iota^*)^{-1}(p_Z^*\mathfrak d)}.
\]
We turn to i) and label the coordinates of 
 \[
   T_X\ =\ \left((\KK^*)^{n_1}\times(\KK^*)^{n_1}\right) \times\ldots\times \left((\KK^*)^{n_\mathbf r}\times(\KK^*)^{n_\mathbf r}\right)
  \]
by $_rU_{ij}$ where $r=1,\ldots,\mathbf r$, $i=1,\ldots,n_r$ and $j=n_r+1,n_r+2$. Then the comorphism $p_X^*$ of the corresponding Laurent polynomial rings is given as
\[
 p_X^*\colon\ \KK[_rU^\pm_{ij}]\quad \to\quad \KK[_rT^\pm_{ij}];
 \qquad
 _rU_{ij}\quad\mapsto\quad _rT_{ij}\ _rT_{0j}^{-1}.
\]
The vanishing ideal of $\Delta_X\cap T_X$ is generated by
\[
 _rU_{i,n_r+1}\ -\ _rU_{i,n_r+2}\qquad\text{where}\quad r=1,\ldots,\mathbf r,\text{ and }i=1,\ldots,n_r.
\]
Note that for any $r=1,\ldots,\mathbf r$ and $i,i'=1,\ldots,n_r$ this ideal also contains the elements 
\[
 _rU_{i,n_r+1}\ _rU_{i',n_r+1}^{-1}\ -\ _rU_{i,n_r+2}\ _rU_{i',n_r+2}^{-1}.
\]
Pulling back all these equation via $p_X^*$ yields the ideal $\iota_X^*(\mathfrak a_X)$ in the Laurent polynomial ring $\mathcal O(T_{\hat X})$.
Since $\mathfrak a_X$ is an ideal of $2\times2$-minors, it is prime (in fact, it is the vanishing ideal of the Segre embedding). Hence Remark~\ref{rem:prime_saturation} gives our assertion.

 We turn to ii) and proceed analogously. Here the coordinates of the dense torus $T_Y=(\KK^*)^n$ will be labeled $U_i$ with $i=1,\ldots, n$. The comorphism $p_Y^*$ is given by
 \[
 p_Y^*\colon\ \KK[U^\pm_i]\quad \to\quad \KK[T^\pm_{ij}];
 \qquad
 U_i\quad\mapsto\quad T_{i,n+1}\ T_{i,n+2}^{-1}.
\]
In $\mathcal O(T_Y)$ the ideal of $\Delta_Y\cap T_Y$ is generated by the relations $U_i-U_j$ for $1\le i<j\le n$. Pulling them back via $p_Y^*$ yields the ideal $\iota^*\mathfrak a_Y$ and the same argument as in i) yields the assertion.
\end{proof}

We denote the functions from Lemma~\ref{lem:vanishing_ideals} by $_rf_{ij}\in\mathcal O(\overline X)$ and $g_{ij}\in\mathcal O(\overline Y)$ where $r$ corresponds to the $r$-th matrix and in both cases $i,j$ define the respective columns.
These functions $_rf_{ij}$ and $g_{ij}$ give rise to the stretched embeddings
\begin{align*}
 \overline{\pi}_X\colon\ 
 \bigoplus_{r=1}^\mathbf r\KK^{2(n_r+1)}
 \quad &\to\quad
 \bigoplus_{r=1}^\mathbf r\left(\KK^{2(n_r+1)}\oplus \KK^{{n_r+1}\choose 2}\right)\\
 (x_1,\ldots,x_\mathbf r)
 \quad &\mapsto\quad
 \left(\,(x_1,\ _1f_{ij}(x_1))\,,\,\ldots\,,\,(x_\mathbf r,\ _\mathbf rf_{ij}(x_\mathbf r))\,\right)\\
\\
 \overline{\pi}_Y\colon\ 
 \KK^{2n}
 \quad &\to\quad
 \KK^{2n}\oplus \KK^{n\choose 2}\\
 y
 \quad &\mapsto\quad
 (y,\ g_{ij}(y)).
\end{align*}
The vanishing ideals of the images are given by
\begin{align*}
 I_{X,1}
 \ &:=\ 
 \langle\,_rT_{ij}\ -\ _rf_{ij};~r=1,\ldots,\mathbf r,~0\le i<j\le n_r+2,\ i\le n_r\,\rangle,\\
  I_{Y,1}
 \ &:=\ 
 \langle\,S_{ij}\ -\ g_{ij};~1\le i<j\le n+2,\ i\le n\,\rangle.
\end{align*}
We denote by $\iota_{X,1}^*,\;\iota_{X,2}^*,\;\alpha_X^*$ and $\iota_{Y,1}^*,\;\iota_{Y,2}^*,\;\alpha_Y^*$ the respective morphisms from Construction~\ref{con:transfer_ideals}. The new Laurent polynomial rings are then given by
\[\KK[T_\infty^{\pm},\ _rT_{ij}^\pm;~r=1,\ldots,\mathbf r,~0\le i<j\le n_r+2,~i\le n_r],\]
\[\KK[S_{ij};~1\le i<j\le n+2],\]
where the additional variables are $T_\infty$ and $S_{n+1,n+2}$ respectively. We transfer the above ideals according to Construction~\ref{con:transfer_ideals}, i.e. we set 
\begin{align*}
  I'_{X,2}
  :&=
  \langle\,\alpha_X^*\ (\iota_{X,1}^*(I_{X,1}))\ \rangle\\
  &=
  \langle\,_rT_{ij}T_\infty\ -\ _rf_{ij};~r=1,\ldots,\mathbf r,~0\le i<j\le n_r+2,\ i\le n_r\,\rangle,\\\\
 I'_{Y,2}
  :&=
  \langle\,\alpha_Y^*\ (\iota_{Y,1}^*(I_{Y,1}))\ \rangle\\
  &=
  \langle\,S_{ij}S_{n+1,n+2}\ -\ g_{ij};~1\le i<j\le n+2,\ i\le n\,\rangle.  
\end{align*}

We first have to compute their preimages under $\iota^*_{X,2}$ and $\iota^*_{Y,2}$, then we then show that $T_\infty$ and $S_{n+1,n+2}$ define prime elements and divide non of the remaining variables. Since the resulting relations are very closely related to the Pl\"ucker relations, we introduce some new notation. For this let $0\le i,j,k,l\le n$ be distinct integers. Then we denote by $q(i,j,k,l)$ the corresponding Pl\"ucker relation; i.e. if $i<j<k<l$ holds, then we set
\[
 q(i,j,k,l)\quad :=\quad T_{ij}T_{kl}-T_{ik}T_{jl}+T_{il}T_{jk}\ \in\ \KK[T_{ij};~0\le i<j\le n].
\]

\begin{lemma}
\label{lem:more_plueckers}
 Let $0\le i_0,j_0\le n$ be distinct integers. In the Laurent polynomial ring $\KK[T_{ij}^\pm;~0\le i<j\le n]$ consider the ideal 
 \[
  I\ :=\ \big\langle\, q\,(i_0,j_0,k,l);\quad 0\le k,l\le n,\quad i_0,j_0,k,l\text{ pairwise distinct}\,\big\rangle.
 \]
Then for any pairwise distinct $0\le i,j,k,l\le n$ we have $q(i,j,k,l)\in I$.
\end{lemma}

\begin{proof}
We first claim that for distinct $0\le i,j,k,l,m\le n$ we have 
 \[
 (*)\quad q(i,j,k,l),\ q(i,j,k,m),\ q(i,j,l,m)\,\in\, I
 \quad \Longrightarrow\quad
 q(i,k,l,m)\,\in\, I.
 \]
 For this we assume without loss of generality that $i<j<k<l<m$ holds. The claim then follows from the relation
 \[
  q(i,k,l,m)
  \ =\ 
  \frac{T_{jk}}{T_{ij}}\,q(i,j,l,m)
  \ -\ 
  \frac{T_{jl}}{T_{ij}}\,q(i,j,k,m)
  \ +\ 
  \frac{T_{jm}}{T_{ij}}\,q(i,j,k,l)\ \in\ I.
 \]
Now consider distinct $0\le\alpha,\beta,\gamma,\delta\le n$. If $\{\alpha,\beta,\gamma,\delta\}\cap\{i_0,j_0\}\not=\emptyset$ holds, then $q(\alpha,\beta,\gamma,\delta)\in I$ follows from the above claim $(*)$. So assume that $\{\alpha,\beta,\gamma,\delta\}$ and $\{i_0,j_0\}$ are disjoint. Applying $(*)$ to the three collections of indices
\[
 i_0,j_0,\alpha,\beta,\gamma;
 \qquad\quad
 i_0,j_0,\alpha,\beta,\delta;
 \qquad\quad
 i_0,j_0,\alpha,\gamma,\delta
\]
shows that $q(i_0,\alpha,\beta,\gamma)$, $q(i_0,\alpha,\beta,\delta)$ and $q(i_0,\alpha,\gamma,\delta)$ lie in $I$. Another application of $(*)$ then proves $q(\alpha,\beta,\gamma,\delta)\in I$.
 \end{proof}

 We are now ready to prove Theorem~\ref{thm:main2}, for Theorem~\ref{thm:main1} we require some further preparations.
\begin{proof}[Proof of Theorem~\ref{thm:main2}]
 Using Lemma~\ref{lem:more_plueckers} we see that the ideals $\langle\, \iota^*_{Y,2}I_Y\, \rangle$ and $I_{Y,2}'$ coincide. Since $I_Y$ is prime from Remark~\ref{rem:prime_saturation} we infer that
 \[
  (\iota_{Y,2}^*)^{-1}I_{Y,2}'
  \ =\ 
  (\iota_{Y,2}^*)^{-1}\langle\, \iota^*_{Y,2} I_Y\, \rangle
  \ =\ 
  I_Y.
 \]
Since $I_Y$ is the ideal of Pl\"ucker relations, $S_{n+1,n+2}$ is prime and does not divide any of the remaining variables. We determine the grading of the Cox ring. The ring $\mathcal O(\overline Y)=\KK[S_{ij};~i=1,\ldots,n,~j=n+1,n+2]$ is $\ZZ^n$-graded by $\deg (S_{ij})=e_i$. Under the stretched embedding the new variables $S_{ij}$ where $1\le i< j\le n$ are assigned the degrees $\deg(S_{ij})=\deg(f_{ij})=e_i+e_j$. Finally, under the blow-up the weights are modified according to \ref{con:transfer_ideals} to give the asserted grading.
\end{proof}

We turn to the remaining case of $X=X'\times X'$.

\begin{lemma}
\label{lem:radprime_prime}
 Let $R:=\KK[T_{\infty},T_1,\ldots,T_n]$ be graded by $\ZZ_{\ge 0}$ and let $I\subseteq R$ be a homogeneous ideal. Suppose that $T_{\infty}\notin\sqrt I$ and $\deg (T_\infty)>0$ hold. If the ideals $I+\langle T_{\infty}\rangle$ and $\sqrt I$ are prime, then so is $I$.
\end{lemma}

\begin{proof}
Compare also \cite[Proof~of~Theorem~1]{collineations}. Since $I+\langle T_{\infty}\rangle$ is a radical ideal, we have $\sqrt I\subseteq I+\langle T_{\infty}\rangle$. With this we obtain
 \[
  \sqrt I
  \quad =\quad
  (I+\langle T_{\infty}\rangle) \cap \sqrt I
  \quad =\quad 
  I+\langle T_{\infty}\rangle \sqrt I.
 \]
Note that for the second equality we used that $\sqrt I$ is prime and $T_{\infty}\notin \sqrt I$ holds. Let $\pi\colon R\to R/I$ denote the canonical projection of $\ZZ_{\ge 0}$-graded algebras. Then we have $\pi(\sqrt I) =\pi(\langle T_\infty\rangle\,\sqrt I)$ and $\deg (\pi(T_\infty))>0$. The assertion follows from the graded version of Nakayama's Lemma.
\end{proof}

\begin{lemma}[{\cite[Proposition~4]{collineations}}]
\label{lem:Tinfty_prime}
 Let $1\le c\le n$ be an integer. Then in the polynomial ring $\KK[T_{ij};~0\le i< j\le n+2]$ the following relations generate a prime ideal
 \begin{align*}
  -T_{ik}T_{jk}+T_{il}T_{jk};
  &\quad 
  0\le i<j\le c< k<l\le n+2,\\
  T_{ij}T_{kl}-T_{ik}T_{jk}+T_{il}T_{jk};
  &\quad
  0\le i<j\phantom{\le c }\,\,< k<l\le n+2\quad \text{different from above.}
 \end{align*}
\end{lemma}

\begin{proof}[Proof of Theorem~\ref{thm:main2}]
First we claim that the ideal $I_X$ is prime. For this note that the ideal $\langle T_\infty\rangle+I_X$ is generated by $T_\infty$ and the equations
\begin{align*}
  -\ {_rT_{ik}}\ {_rT_{jk}}\ +\  {_rT_{il}}\ {_rT_{jk}};
  &\quad 
  0\le i<j\le n_r,\quad k=n_r+1,\ l= n_r+2,\\
   {_rT_{ij}}\ {_rT_{kl}}\ -\ {_rT_{ik}}\ {_rT_{jk}}\ +\ {_rT_{il}}\ {_rT_{jk}};
  &\quad
  0\le i<j< k<l\le n+2\quad \text{diff. f. above}
 \end{align*}
where $r=1,\ldots,\mathbf r$. From Lemma~\ref{lem:Tinfty_prime} we infer that $\langle T_\infty\rangle+I_X$ is prime; we check the remaining assumptions of Lemma~\ref{lem:radprime_prime}. Consider the classical grading of $R_X$, then $I_X$ is homogeneous and $\deg T_\infty >0$ holds. We only have to verify that $V(I_X)$ is irreducible. For this
recall that we transferred the ideal $I_{X,1}$ via
\[
 I_{X,1}'\ =\ \langle\, \iota_{X,1}^*I_{X,1}\,\rangle
 \qquad\text{and}\qquad
 I_{X,2}'\ =\ \langle\, \alpha_X^*I_{X,1}'\,\rangle.
\]
Treating the index $\infty$ as $n_r+1,n_r+2$ in Lemma~\ref{lem:more_plueckers} we see that the latter ideal is given by $I'_{X,2}=\langle\iota^*_{X,2}(I_X)\rangle$. We track the respective zero sets.
\[
 V(I_X)
 \ =\ 
 \overline{V(I'_{X,2})}
 \ =\ 
 \overline{\alpha_X^{-1} V(I'_{X,1})}
 \ =\ 
 \overline{\alpha_X^{-1}(\iota_{X,1}^{-1}(V(I_{X,1}))}
\]
Since $\alpha_X$ has connected kernel and $V(I_{X,1})$ is the graph of $\overline X$ and as such irreducible, so is $V(I_X)$. This then implies that $I_X$ is prime.

By Remark~\ref{rem:prime_saturation} this means that $I_{X}=(\iota_{X,2}^*)^{-1}\langle\iota_{X,2}^*I_X\rangle=(\iota_{X,2}^*)^{-1}(I'_{X,2})$ holds. By Proposition~\ref{pro:cox_crit} the only thing left to verify is that $T_\infty$ does not divide any of the remaining variables. For this we compute the grading of the Cox ring; for reasons of degree it is then impossible for $T_\infty$ to divide any other variable. The $\ZZ^\mathbf r\times\ZZ^\mathbf r$-grading of 
\[
\mathcal O(\overline X)=\KK[_rT_{ij};~r=1,\ldots,\mathbf r,~i=0,\ldots,n_r,~j=n_r+1,n_r+2]
\]
is given by
\[
 \deg( _rT_{ij})\ =\ \begin{cases}
                      (e_r,0) & \text{if } j=n_r+1,\\
                      (0,e_r) & \text{if } j=n_r+2.
                     \end{cases}
\]
When stretching the embedding we add for every $r=1,\ldots,\mathbf r$ the variables $_rT_{ij}$ where $0\le i<j\le n_r$. These are assigned the degrees $\deg(_rT_{ij})=\deg(_rg_{ij})=(e_r,e_r)$. Finally under the blow-up the degrees are modified according to Construction~\ref{con:transfer_ideals} to give the asserted grading.
\end{proof}

\bibliographystyle{abbrv}

\end{document}